\documentclass{amsart}

\title{On the $K$-theory of truncated polynomial rings in non-commuting variables}
\author{Vigleik Angeltveit}
\address{Department of Mathematics \\
John Dedman Building (Building 27) \\
Australian National University \\
Acton, ACT \\
0200 Australia}

\usepackage{amsxtra}
\usepackage{amsfonts}
\usepackage[latin1]{inputenc}
\usepackage{graphicx}
\usepackage{amsmath,amssymb,latexsym,amsthm,mathrsfs}
\usepackage[all]{xy}
\usepackage{pstricks}
\usepackage{verbatim}

\newtheorem{theorem}{Theorem}[section]
\newtheorem{thm}[theorem]{Theorem}
\newtheorem{lemma}[theorem]{Lemma}

\newtheorem{cor}[theorem]{Corollary}

\newtheorem{prop}[theorem]{Proposition}

\theoremstyle{definition}
\newtheorem{defn}[theorem]{Definition}

\makeatletter
\let\c@equation\c@theorem
\makeatother
\numberwithin{equation}{section}

\newtheorem{lettertheorem}{Theorem}

  \newcommand{\bC}{\mathbb{C}}           \newcommand{\bN}{\mathbb{N}}  
\newcommand{\bQ}{\mathbb{Q}} \newcommand{\bR}{\mathbb{R}}     \newcommand{\bW}{\mathbb{W}}   \newcommand{\bZ}{\mathbb{Z}}

\newcommand{\sma}{\wedge} %smash product
\newcommand{\holim}{\textnormal{holim}}

\newcommand{\coker}{\textnormal{coker}}

\newcommand{\xto}{\xrightarrow}

\newcommand{\TF}{\textnormal{TF}}

\newcommand{\TC}{\textnormal{TC}}
\newcommand{\trc}{\textnormal{trc}}

\begin{document}

\begin{abstract}
We compute the algebraic $K$-theory of the non-commutative ring $k\langle x_1,\ldots,x_n \rangle/(m^a)$ when $k$ is a perfect field of positive characteristic and $m=(x_1,\ldots,x_n)$. We express the answer in terms of the \emph{truncation poset Witt vectors} developed in \cite{An_genWitt}.
\end{abstract}

\maketitle

\section{Introduction}
Let $A = k \langle x_1,\ldots,x_n \rangle/(m^a)$ for some ring $k$, where $k\langle x_1,\ldots,x_n \rangle$ denotes the polynomial ring in $n$ non-commuting variables and $m=(x_1,\ldots,x_n)$ is the kernel of the map to $k$ given by evaluating at $(0,\ldots,0)$.

Hesselholt and Madsen computed the algebraic $K$-theory of $A$ when $k$ is a perfect field of positive characteristic and $n = 1$ in terms of big Witt vectors in \cite{HeMa97b}, but see \cite{HeMa97a} as well. They found that
\[
 K_{2q-1}(k[x]/(x^a), (x)) \cong \coker \big( V_a : \bW_q(k) \to \bW_{aq}(k) \big),
\]
while
\[
 K_{2q}(k[x]/(x^a), (x)) = 0.
\]
Here $\bW_n(k)$ denotes the Witt vectors on the truncation set $\{1,2,\ldots,n\}$.

This was generalized in \cite{AGHL14} to the ring $A=k[x_1,\ldots,x_n]/(x_1^{a_1},\ldots,x_n^{a_n})$, with the big Witt vectors replaced by certain generalized Witt vectors built from truncation sets in $\bN^n$ instead of $\bN$, and the cokernel of $V_a$ replaced by the iterated homotopy cofiber of an $n$-cube of spectra.

These calculations both use the \emph{cyclotomic trace map} from algebraic $K$-theory to topological cyclic homology. The underlying reason for the appearance of Witt vectors is that if $k$ is a perfect field of positive characteristic then
\[
 \lim_{R,m \leq n} \pi_* THH(k)^{C_m} \cong \bW_n(k)[\mu_0],
\]
where $\mu_0$ is a polynomial generator in degree $2$.

Lindenstrauss and McCarthy \cite{LiMc} have recently computed the algebraic $K$-theory of $A$ in the case $a=2$ using different techniques. Note that when $a=2$ the ring $A$ is commutative and can equally well be defined as $k[x_1,\ldots,x_n]/(m^2)$.

In this paper we compute the algebraic $K$-theory of $A = k \langle x_1,\ldots,x_n \rangle/(m^a)$ for any $a$, under the assumption that $k$ is a perfect field of positive characteristic. The answer is most easily expressed using the language truncation poset Witt vectors from \cite{An_genWitt}. To give a very brief summary, if $S$ is a partially ordered set with a function $|-| : S \to \bN$ satisfying some mild conditions, we say $S$ is a \emph{truncation poset}. We can then define the Witt vectors $\bW_S(k)$ in a similar way as for ordinary truncation sets. Maps of truncation posets encode structure maps that generalize the classical structure maps on Witt vectors.

To state the main result we need two definitions.

\begin{defn} \label{d:ourtrunsets}
Fix positive integers $n$, $a$ and $N$. (We allow $N=\infty$.) Let $S_n(a,N)$ denote the following generalized truncation set. As a set, $S_n(a,N)$ is the set of words in $x_1,\ldots,x_n$ of length at most $N$ with length divisible by $a$, modulo the equivalence relation $\sim_a$ given by cyclically permuting blocks of $a$ letters. The map $|-| : S_n(a,N) \to \bN$ is given by defining $|w|$ to be the largest $e$ such that $w=(w')^e$ for some $w' \in S_n(a,N)$.
\end{defn}

\begin{defn} \label{d:ourVmap}
If $b \mid a$, let
\[
 v_a^b : S_n(a,N) \to S_n(b,N)
\]
be the map of generalized truncation sets that sends $[w]_{\sim_a}$ to $[w]_{\sim_b}$, and let
\[
 V_a^b : \bW_{S_n(a,N)}(k) \to \bW_{S_n(b,N)}(k)
\]
be the corresponding additive map of Witt vectors. We will refer to $V_a^b$ as a generalized Verschiebung map.
\end{defn}

See Section \ref{s:genWitt} for a quick definition of $V_a^b$, or \cite{An_genWitt} for a more thorough discussion.

\begin{lettertheorem} \label{t:main}
Suppose $k$ is a perfect field of positive characteristic. The algebraic $K$-theory of $A = k\langle x_1,\ldots,x_n \rangle/m^a$ relative to $k$ is given by the exact sequence
\[
 0 \to K_{2q}(A,m) \to \bW_{S_n(a,aq)}(k) \xto{V_a^1} \bW_{S_n(1,aq)}(k) \to K_{2q-1}(A,m) \to 0
\]
for each $q \geq 1$.
\end{lettertheorem}

For any truncation poset $S$ we can write $\bW_S(k)$ as a product of ordinary Witt vectors. In this case we get a product indexed by equivalence classes of irreducible words. But the generalized Verschiebung map becomes more complicated. We state the result in terms of classical Witt vectors for comparison.

\begin{cor} \label{c:main2}
Suppose $k$ is a perfect field of positive characteristic. The algebraic $K$-theory of $A = k\langle x_1,\ldots,x_n \rangle/(m^a)$ relative to $k$ is given by the following.
\begin{itemize}
 \item In odd degree we have
 \[
  K_{2q-1}(A,m) \cong \bigoplus_w \coker \big( V_{a/g} : \bW_{\lfloor gq/\ell(w) \rfloor}(k) \to \bW_{\lfloor aq/\ell(w) \rfloor}(k) \big).
 \]
 Here the direct sum is over all irreducible cyclic words $w$ in $x_1,\ldots,x_n$, $\ell(w)$ denotes the length of $w$, and $g=\gcd(a,\ell(w))$.
 \item In even degree we have
 \[
  K_{2q}(A,m) \cong \bigoplus_w \bigoplus_{g-1} \bW_{\lfloor gq/\ell(w) \rfloor}(k),
 \]
 where once again the direct sum is over all irreducible cyclc words $w$.
\end{itemize}
\end{cor}

Lindenstrauss and McCarthy \cite{LiMc} state their calculation in the case $a=2$ in terms of $p$-typical Witt vectors, where $p$ is the characteristic of $k$. In that case the answer breaks up even further.

Note that Theorem \ref{t:main} recovers the calculation of Hesselholt and Madsen \cite{HeMa97b}, because when $n=1$ the generalized truncation sets are given by
\[
 S_1(a,aq) = \{a,2a,\ldots,aq\} \cong \{1,2,\ldots,q\}
\]
and
\[
 S_1(1,aq) = \{1,2,\ldots,aq\}.
\]
Hence $\bW_{S_1(a,aq)}(k) \cong \bW_q(k)$ and $\bW_{S_1(1,aq)}(k) \cong \bW_{aq}(k)$. It also follows from the definition that in this case $V_a^1$ is the usual Verschiebung map $V_a$.

\subsection*{Proof strategy}
We will use the cyclotomic trace map
\[
 \trc : K(A,m) \to \TC(A,m),
\]
from \cite{BoHsMa93} which by \cite{Mc97} is an equivalence after profinite completion. In fact, profinite completion turns out to be unneccessary, see \cite[Theorem 7.0.0.3]{DuGoMc13}. Then we will write $\TC(A,m)$ as the cofiber of a map
\[
 \TC'(A,m) \to \TC''(A,m),
\]
where the homotopy groups of $\TC'(A,m)$ and $\TC''(A,m)$ are the groups in the source and target of $V_a^1$.

One cruicial input is the equivariant homotopy type of the space $\Delta^{s-1}/C_s \cdot \Delta^{s-a}$. To prove Theorem \ref{t:main} we first establish a conjecture due to Hesselholt and Madsen which describes this homotopy type. This result might be of independent interest.

\subsection{Organization}
We start with a brief discussion of Witt vectors defined in terms of truncation posets in Section \ref{s:genWitt}, and discuss the particular truncation posets that come up in the calculations.

In Section \ref{s:polytopes} we discuss an important family of simplicial sets and prove a conjecture of Hesselholt and Madsen regarding their homotopy type. In Section \ref{s:monoids} we discuss the cyclic bar construction on the relevant pointed monoid, and in Section \ref{s:THHTFTC} we put everything together to prove Theorem \ref{t:main}.

Finally, in Section \ref{s:integers} we calculate the rational $K$-groups when $k=\bZ$.

\subsection{Acknowledgements}
The author would like to thank Ayelet Lindenstrauss and Lars Hesselholt for interesting conversations. This project got started while the author visited MSRI for a semester program in Algebraic Topology. The author was also supported by an Australian Research Council Discovery grant.

\section{Truncation poset Witt vectors} \label{s:genWitt}
We will very briefly review the construction of truncation poset Witt vectors from \cite{An_genWitt}. Recall that the ``classical'' Witt vectors can be defined for a truncation set $S \subset \bN=\{1,2,\ldots\}$ and that the \emph{length $n$ big Witt vectors} of $k$, denoted $\bW_n(k)$, are defined to be $\bW_S(k)$ for $S=\{1,2,\ldots,n\}$. The ring structure on $\bW_S(k)$ is defined in such a way that the ghost map $\bW_S(k) \to k^S$ is a ring map, functorially in $k$.

A \emph{truncation poset} is a partially ordered set $S$ with a map $|-| : S \to \bN$ satisfying axioms which imply that there is an isomorphism $S \cong \coprod S_i$ with each $S_i$ an ordinary truncation set. The ring $\bW_S(k)$ of $S$-Witt vectors is then isomorphic to the product of the rings $\bW_{S_i}(k)$. The power of this point of view comes from the fact that all the usual structure maps of Witt vectors can be encoded by maps of truncation posets, and that some maps are much easier to describe using this language.

We will be interested in certain truncation posets defined in terms of words in $x_1,\ldots,x_n$.

\begin{defn}
Let $\sim_a$ denote the equivalence relation on words of length divisible by $a$ given by cyclically permuting blocks of length $a$.
\end{defn}

If $w_1$ and $w_2$ are in the same equivalence class we say $w_1$ and $w_2$ are $a$-equivalent. If $a=1$ we call an equivalence class of words a \emph{cyclic word}. We can make the set of equivalence classes of words into a truncation poset as follows.

\begin{defn}
Given two $a$-equivalence classes of words $[w_1]$ and $[w_2]$, we say $[w_1]$ divides $[w_2]$ if there are representatives $w_1$ and $w_2$ with $w_2=w_1^e$ for some $e$.

If $S$ is a set of $a$-equivalence classes of words which is closed under division, we define $|[w]|$ to be the largest $e$ such that $[w]$ is divisible by $e$.
\end{defn}

For example, the words $w_1=x_1x_2x_1x_2$ and $w_2=x_2x_1x_2x_1$ are $1$-equivalent but not $2$-equivalent or $4$-equivalent. If $a=1$ or $a=2$ we have $|w_1|=|w_2|=2$, while if $a=4$ we have $|w_1|=|w_2|=1$.

With these definitions it is clear how to interpret the truncation posets in Definition \ref{d:ourtrunsets}.

If $f : S \to T$ is a map of truncation posets which satisfies some mild conditions (a $T$-map, in the language of \cite{An_genWitt}) we get an induced map $f_\oplus : \bW_S(k) \to \bW_T(k)$ which is given on ghost coordinates by
\[
 (f_\oplus^w \langle x_s \rangle)_t = \sum_{f(s)=t} \frac{|t|}{|s|} x_s.
\]
This is the map $V_a^b$ in Definition \ref{d:ourVmap}.

Note that the map $v_a^b : S_n(a,N) \to S_n(b,N)$ is neither injective nor surjective in general. For example, if $a=2$ and $b=1$ the words $w_1=x_1 x_2$ and $w_2=x_2 x_1$ are different in $S_n(2,2)$ while $[w_1]=[w_2]$ in $S_n(1,2)$. As a result, the induced map $V_a^b : \bW_{S_n(a,N)}(k) \to \bW_{S_n(b,N)}(k)$ will be neither injective nor surjective in general.

The next result is almost a tautology.

\begin{lemma}
The truncation poset $S_n(a,N)$ splits as
\[
 S_n(a,N) = \coprod_w S_n(a,N)[w],
\]
where the coproduct is over irreducible words in $S_n(a,N)$ and $S_n(a,N)[w]$ consists of all powers of $w$ in $S_n(a,N)$. Moreover, each $S_n(a,N)[w]$ is isomorphic to the ordinary truncation set $\{1,2,\ldots,\lfloor \frac{N}{\ell(w)} \rfloor\}$, with the isomorphism given by mapping $w^e$ to $e \in \bN$. Here $\ell(w)$ is the length of the word $w$. Hence
\[
 \bW_{S_n(a,N)}(k) \cong \prod_w \bW_{\lfloor \frac{N}{\ell(w)} \rfloor}(k),
\]
where the product is over irreducible words in $S_n(a,N)$.
\end{lemma}

Using this it is not hard to see that Corollary \ref{c:main2} is a restatement of Theorem \ref{t:main}.

\section{Cyclic polytopes and a conjecture of Hesselholt and Madsen} \label{s:polytopes}
When Hesselholt and Madsen \cite{HeMa97b} computed the algebraic $K$-theory of $k[x]/(x^a)$, they had to understand the $S^1$-equivariant homotopy type of
\[
 S^1_+ \sma_{C_s} (\Delta^{s-1}/C_s \cdot \Delta^{s-a}).
\]
Here $\Delta^{s-a} \subset \Delta^{s-1}$ is the subspace spanned by the first $s-a+1$ vertices, and $C_s \cdot \Delta^{s-a} = \bigcup_{i=0}^{s-1} t^i \Delta^{s-a}$, with $t$ cyclically permuting the vertices of $\Delta^{s-1}$. They found \cite[Theorem B]{HeMa97b} that
\[
 S^1_+ \sma_{C_s} (\Delta^{s-1}/C_s \cdot \Delta^{s-a}) \simeq_{S^1} \begin{cases} S^1_+ \sma_{C_s} S^{V_d} \quad & \textnormal{if $da < s < (d+1)a$} \\ S^1_+ \sma_{C_s} (C_a * S^{V_d}) \quad & \textnormal{if $s=(d+1)a$} \end{cases}
\]
Here $V_d$ is the $d$-dimensional complex $S^1$-representation $\bC(1) \oplus \ldots \oplus \bC(d)$. To prove Theorem \ref{t:main} we will need to understand the $S^1$-equivariant homotopy type of $S^1_+ \sma_{C_e} (\Delta^{s-1}/C_s \cdot \Delta^{s-a})$ for all $e \mid s$. The easiest way to do that is to prove the following result, conjectured by Hesselholt and Madsen:

\begin{thm}
We have
\[
 \Delta^{s-1}/C_s \cdot \Delta^{s-a} \simeq_{C_s} \begin{cases} S^{V_d} \quad & \textnormal{if $da < s < (d+1)a$} \\ C_a * S^{V_d} \quad & \textnormal{if $s=(d+1)a$} \end{cases}
\]
\end{thm}

Hesselholt and Madsen used that the homology of $S^1_+ \sma_{C_s} (\Delta^{s-1}/C_s \cdot \Delta^{s-a})$ can be interpreted as a direct summand of the Hochschild homology of $\bZ[x]/(x^a)$, and the Buenos Aires Group had already determined $HH_*(\bZ[x]/(x^a))$ in \cite{BAG91}. There is no obvious homological algebra interpretation of the homology of $\Delta^{s-1}/C_s \cdot \Delta^{s-a}$, so the techniques from \cite{BAG91} do not apply. Instead we give a combinatorial proof.

\begin{defn}
Let $X_{s,a}$ denote the pointed simplicial set generated by an $(s-1)$-simplex $x_0 \otimes x \otimes \ldots \otimes x$ ($s$ tensor factors) and a basepoint $*$, with face and degeneracy maps as in the Hochschild chain complex and with relations $x^a=*$ and $x^{i-1} x_0 x^{a-i}=*$ for $1 \leq i \leq a$.
\end{defn}

Here we distinguish the zero'th tensor factor from the others to force
\[
  d_0(x_0 \otimes x \otimes \ldots \otimes x) = x_0 x \otimes x \otimes \ldots \otimes x
\]
to be different from
\[
 d_s(x_0 \otimes x \otimes \ldots \otimes x) = x x_0 \otimes x \otimes \ldots \otimes x.
\]

\begin{lemma}
The geometric realization of $X_{s,a}$ is homeomorphic to $\Delta^{s-1}/C_s \cdot \Delta^{s-a}$.
\end{lemma}

\begin{proof}
This is formal, because $X_{s,a}$ is generated by a single $(s-1)$-simplex and the relations collapses $t^i \Delta^{s-a}$ for $0 \leq i \leq s-1$ to a point. 
\end{proof}

Let $C_*(X_{s,a})$ be the simplicial chain complex of $X_{s,a}$, generated by the non-degenerate simplices. Then we can write an element $\alpha \in C_e(X_{s,a})$ as a sum
\[
 \alpha = \sum_{\vec{k}} c_{\vec{k}} \cdot x^{k_0'} x_0 x^{k_0''} \otimes x^{k_1} \otimes \ldots \otimes x^{k_e}
\]
over vectors $\vec{k} = (k_0', k_0'', k_1,\ldots,k_e)$ with the following conditions on $\vec{k}$:
\begin{enumerate}
\item We have $k_0' + k_0'' + k_1 + \ldots + k_e = s-1$.
\item We have $k_0'+1+k_0''<a$ and $1 \leq k_i < a$ for all $1 \leq i \leq e$.
\end{enumerate}

\noindent
Let $d = \lfloor \frac{s-1}{a} \rfloor$.

\begin{prop} \label{p:poly}
The reduced homology of $X_{s,a}$ is given as follows.
\begin{enumerate}
\item \label{eq:poly_part1} If $a \nmid s$ the only non-trivial reduced homology of $X_{s,a}$ is a $\bZ$ in degree $2d$, represented by the cycle
\[
 \iota_{2d} = \sum_{k_0'' + k_2 + k_4 + \ldots + k_{2d} = s-d-1} x_0 x^{k_0''} \otimes x \otimes x^{k_2} \otimes x \otimes \ldots \otimes x \otimes x^{k_{2d}}.
\]
\item \label{eq:poly_part2} If $a \mid s$ the only non-trivial reduced homology of $X_{s,a}$ is $\bZ^{a-1}$ in degree $2d+1$, represented by the cycles
\[
 \iota_{2d+1,i} = x^i x_0 x^{a-i-2} \otimes x \otimes x^{a-1} \otimes x \otimes \ldots \otimes x^{a-1} \otimes x
\]
for $0 \leq i \leq a-2$.
\end{enumerate}

\end{prop}

\begin{proof}
It follows from the proof of Theorem \ref{t:thetaequiv} below that $\iota_{2d}$ and $\iota_{2d+1,i}$ are linearly independent in $H_*(X_{s,a})$ so it suffices to prove that $H_*(X_{s,a})$ does not have any other summands.

We prove both parts simultaneously. The case $a=2$ is clear, so we will assume $a > 2$. Let
\[
 \alpha = \sum_{\vec{k}} c_{\vec{k}} \cdot x^{k_0'} x_0 x^{k_0''} \otimes x^{k_1} \otimes \ldots \otimes x^{k_e}
\]
be a cycle, and suppose $\alpha$ represents a non-zero element in $H_e(X_{s,a})$. Then we want to show that $\alpha$ is homologous to a multiple of $\iota_{2d}$ or a linear combination of $\iota_{2d+1,i}$ for $0 \leq i \leq a-2$.
\newline

\noindent
(A) We modify $\alpha$ so that $k_1=1$. Let
\[
 \beta_{1,1} = \sum_{\vec{k} \textnormal{ with } k_1>1} c_{\vec{k}} \cdot x^{k_0'} x_0 x^{k_0''} \otimes x^{k_1-1} \otimes x \otimes x^{k_2} \otimes \ldots \otimes x^{k_e}
\]
and let $\alpha_{1,1} = \alpha+d(\beta_{1,1})$. Then $\alpha_{1,1}$ is a sum over $\vec{k}$ with $k_1 \leq a-2$. Repeating the process with $\beta_{1,2},\ldots,\beta_{1,a-2}$ then produces a representative
\[
 \alpha_1 = \sum_{\vec{k} \textnormal{ with } k_1=1} c^1_{\vec{k}} \cdot x^{k_0'} x_0 x^{k_0''} \otimes x \otimes x^{k_2} \otimes \ldots \otimes x^{k_e}.
\]
\newline

\noindent
(B) Now we make the cruicial observation that in order for $\alpha_1$ to be a cycle, the terms in $d(\alpha_1)$ have to cancel in a very particular way. Suppose $k_2 \leq a-2$, so that $x^{1+k_2} \neq 0$. The only way to cancel a term
\[
 d_1(x^{k_0'} x_0 x^{k_0''} \otimes x \otimes x^{k_2} \otimes \ldots \otimes x^{k_e}) = -x^{k_0'} x_0 x^{k_0''} \otimes x^{1+k_2} \otimes \ldots \otimes x^{k_e}
\]
is from a corresponding term
\[
 d_0(x^{k_0'} x_0 x^{k_0''-1} \otimes x \otimes x^{1+k_2} \otimes \ldots \otimes x^{k_e}).
\]
This term only exists if $k_0'' \geq 1$; if $k_2 \leq a-2$ and $k_0''=0$ we reach a contradiction. If $1+k_2 \leq a-2$ we can repeat the argument, and it follows that $k_0'' + k_2 \geq a-1$ and that
\begin{equation}
 c^1_{k_0',k_0'', k_2,\ldots,k_e} = c^1_{k_0', k_0''-i, k_2+i,\ldots,k_e} \label{eq:shuffle}
\end{equation}
for every $i \geq 0$ with $k_2+i \leq a-1$.

Next suppose $k_0'+1+k_0''+1 \leq a-1$ and consider how to cancel the term
\[
 d_0(x^{k_0'} x_0 x^{k_0''} \otimes x \otimes x^{k_2} \otimes \ldots \otimes x^{k_e}) = x^{k_0'} x_0 x^{k_0''+1} \otimes x^{k_2} \otimes \ldots x^{k_e}.
\]
If follows that $k_2 \geq 2$ and that this term can only cancel with
\[
 d_1(x^{k_0'} x_0 x^{k_0''+1} \otimes x \otimes x^{k_2-1} \otimes \ldots \otimes x^{k_e}).
\]
Hence Equation \ref{eq:shuffle} holds for negative $i$ as well, as long as $k_0'+1+k_0''+i \leq a-1$.

It then follows that
\[
 d_0(\alpha_1) - d_1(\alpha_1) = 0,
\]
so
\[
 d(\alpha) = \sum_{i=2}^e (-1)^i d_i(\alpha_1).
\]

\noindent
Now we repeat the above two steps. The analogue of (A) is as follows. Let
\[
 \beta_{3,1} = \sum_{\vec{k} \textnormal{ with } k_3>1} c^1_{\vec{k}} x^{k_0'} x_0 x^{k_0''} \otimes x \otimes x^{k_2} \otimes x^{k_3-1} \otimes x \otimes \ldots \otimes x^e
\]
and let $\alpha_{3,1} = \alpha_1 + d(\beta_{3,1})$. Repeating the process produces a representative
\[
 \alpha_3 = \sum_{\vec{k} \textnormal{ with } k_1=k_3=1} c^3_{\vec{k}} \cdot x^{k_0'} x_0 x^{k_0''} \otimes x \otimes x^{k_2} \otimes x \otimes x^{k_4} \otimes \ldots \otimes x^{k_e}.
\]
Note that we still have $d_0(\alpha_3)-d_1(\alpha_3)=0$.

The analogue of (B) is as follows. Because only $d_2$ and $d_3$ produce elements with $k_3 \geq 2$, it follows that there is only one way a term $d_3(-)$ can cancel, It follows that if $k_4 \leq a-2$ then $c^3_{k_0', k_0'', k_2,k_4,\ldots,k_e} = c^3_{k_0',k_0'', k_2-1, k_4+1,\ldots,k_e}$. As before the only way to avoid a contradiction is to have $k_2+k_4 \geq a$. We can also conclude that $d_2(\alpha_3)-d_3(\alpha_3)=0$, so $d(\alpha_3) = \sum_{i=4}^e (-1)^i d_i(\alpha_3)$.

Suppose $e=2d'$ is even. Then we can repeat (A) and (B) to produce a representative $\alpha_{2d'-1}$, and finally we can modify $\alpha_{2d'-1}$ by defining
\[
 \beta = \sum_{\vec{k}} c^{2d'-1}_{\vec{k}} \cdot x_0 x^{k_0''} \otimes x \otimes \ldots \otimes x^{k_{2d'}} \otimes x^{k_0'}
\]
and setting $\alpha_{2d'} = \alpha_{2d'-1}-d(\beta)$. It follows that $\alpha_{2d'}$ is given by a sum
\[
 \alpha_{2d'} = \sum_{\vec{k}} c \cdot x_0 x^{k_0''} \otimes x \otimes x^{k_2} \otimes \ldots \otimes x \otimes x^{k_{2d'}}.
\]
If $d'=d$ this is $c$ times the cycle given in Part (\ref{eq:poly_part1}). If $d' < d$ we get a contradiction because the degree of $\alpha_{2d'}$ (which is equal to $s$) is at most $ad'$, but $ad' < ad \leq s-1$. If $d > d'$ we get a contradiction by considering the extreme case $k_2=k_4=\ldots=k_{2d'}=a-1$. It follows that $k_0''=s-d'a-1$, which is negative.

Now suppose $e=2d'+1$ is odd. Then we can repeat (A) and (B) to produce a representative $\alpha_{2d'-1}$ where $c^{2d'-1}_{\vec{k}}$ depends only on $k_0'$, and finally we can modify $\alpha_{2d'-1}$ by defining
\[
 \beta = \sum_{\vec{k} \textnormal{ with } k_{2d'+1} > 1} c^{2d'-1}_{\vec{k}} \cdot x^{k_0'} x_0 x^{k_0''} \otimes \ldots \otimes x \otimes x^{k_{2d'+1}-1}
\]
and $\alpha_{2d'+1} = \alpha_{2d'-1} - d(\beta)$. Then
\[
 \alpha_{2d'+1} = \sum_{\vec{k}} c^{2d'+1}_{\vec{k}} \cdot x^{k_0'} x_0 x^{k_0''} \otimes x \otimes x^{k_2} \otimes x \otimes \ldots \otimes x^{k_{2d'}} \otimes x,
\]
where $c^{2d'+1}_{\vec{k}}$ depends only on $k_0'$. If $d'=d$ then $k_0'+1+k_0'' = a-1$ and $k_{2i}=a-1$ for all $1 \leq i \leq 2d$, and $\alpha_{2d+1}$ is a linear combination of the cycles given in Part (\ref{eq:poly_part2}). If $d'<d$ we get a contradiction because the degree of $\alpha_{2d'+1}$ is too small, and if $d'>d$ we get a contradiction by considering the extreme case $k_2=k_4=\ldots=k_{2d'}=a-1$.
\end{proof}

We need to recall some things from \cite{HeMa97b}. First suppose $a \nmid s$ and let $d= \lfloor \frac{s-1}{a} \rfloor$ as above. By standard representation theory there is a projection $\pi : \bR C_s \to V_d$. Hesselholt and Madsen let $P_{s,d}=\pi(\Delta^{s-1})$ and $Q_{s,d}=\pi(C_s \cdot \Delta^{s-a})$, so that there is an induced map $\pi_d : \Delta^{s-1}/C_s \cdot \Delta^{s-a} \to P_{s,d}/Q_{s,d}$.

There is another map $r : P_{s,d}/Q_{s,d} \to P_{s,d}/\partial P_{s,d} = S^{V_d}$ so we get a map
\[
 \theta_{s,a} = r \circ \pi_d : \Delta^{s-1}/C_s \cdot \Delta^{s-a} \to S^{V_d}.
\]

If $a \mid s$, they consider $\pi_{d+1} : \Delta^{s-1}/C_s \cdot \Delta^{s-a} \to P_{s,d+1}/Q_{s,d+1}$ instead, and another map $r : P_{s,d+1}/Q_{s,d+1} \to C_n * S^{V_d}$ so in this case we get a map
\[
 \theta_{s,a} = r \circ \pi_{d+1} : \Delta^{s-1}/C_s \cdot \Delta^{s-a} \to C_n * S^{V_d}.
\]

The join $C_n * S^{V_d}$ can be thought of as the union of $n$ copies of the cone $C S^{V_d}$ along $S^{V_d}$. Non-equivariantly this is homotopy equivalent to a wedge of $n-1$ copies of $\Sigma S^{2d} \cong S^{2d+1}$. Hence we see that the source and target of $\theta_{s,a}$ have isomorphic homology groups.

\begin{thm} \label{t:thetaequiv}
The map $\theta_{s,a}$ is a $C_s$-equivariant homotopy equivalence both when $a \nmid s$ and when $a \mid s$.
\end{thm}

The proof relies heavily on Hesselholt and Madsen's paper \cite{HeMa97b}, especially Section 3. Rather than repeating that entire section we assume the reader is familiar with it and explain the minor modifications necessary to prove this stronger result.

\begin{proof}
It suffices to prove that $\theta_{s,a}$ induces a non-equivariant homotopy equivalence on $C_e$-fixed points for all $e \mid s$. But the $C_e$-fixed points of $\theta_{s,a}$ is $\theta_{s/e,a}$, so it follows that it suffices to prove that $\theta_{s,a}$ induces a non-equivariant homotopy equivalence for all $s$.

If $d > 0$ the source and target of $\theta_{s,a}$ are both simply-connected, so it suffices to show that $\theta_{s,a}$ induces an isomorphism on reduced homology. If $d=0$ a direct geometric argument can be used instead.

The case when $a \nmid s$ follows as in \cite[Lemma 3.2.5]{HeMa97b}. The only thing missing in \emph{loc.\ cit.\ }to conclude is the fact that $\Delta^{s-1}/C_s \cdot \Delta^{s-a}$ does not have any more homology, which we proved in Proposition \ref{p:poly} above.

To conclude in the case when $a \mid s$, we use a slightly stronger version of \cite[Lemma 3.3.8]{HeMa97b}. Instead of the single chain $c : \Delta^{2d+1} \to \Delta^{s-1}$ we use the $a-2$ chains $c_{2d+1,i} : \Delta^{2d+1} \to \Delta^{s-1}$ defined by
\[
 c_{2d+1,i} = (d^0)^{a-i-2} (d^2)^{a-2} \cdots (d^{2d})^{a-2} (d^{2d+2})^i.
\]
These are obtained by dualizing the face maps needed to go from $x_0 \otimes x \otimes \ldots \otimes x$ to the generators of $\widetilde{H}_{2d+1}(X_{s,d})$. (Note that there is a typo in the definition of $c$ in \cite[p.\ 86]{HeMa97b}, the exponents should be $n-2$ rather than $n-1$.)

The point is that the proof of \cite[Lemma 3.3.8]{HeMa97b} actually shows that $\theta_{s,a}(c)$ is a particular generator of $\widetilde{H}_{2d+1}(C_n * S^{V_d})$. This is because $\partial \theta_{s,a}(c)$ in $C_{2d}(S(V_d))$ intersects $\textbf{n}$ in a single point. (Here $\textbf{n}$ is the set of midpoints in the regular $n$-gon in $\bC(\xi_s^{d+1})$ as defined near the top of \cite[p.\ 86]{HeMa97b}.) But each of the $n-2$ chains $\partial \theta_{s,a}(c_{2d+1,i})$ intersect $\textbf{n}$ in a different point, so the result follows.
\end{proof}

\section{Pointed monoids} \label{s:monoids}
Fix integers $n \geq 1$ and $a \geq 2$. We will consider the pointed monoid
\[
 \Pi_{a,n} = \{0\} \cup \{w \quad | \quad \textnormal{$w$ is a word in $x_1,\ldots,x_n$ of length $\ell(w)<a$} \}.
\]
Here multiplication is given by concatenation of words, and two words multiply to $0$ if their length adds up to $a$ or more. We write $1$ for the empty word. For example,
\[
 \Pi_{2,n} = \{0,1,x_1,\ldots,x_n\}.
\]
Given a pointed monoid $\Pi$ and a ring $k$ we are interested in the pointed monoid algebra $k(\Pi)$. We get
\[
 k(\Pi_{a,n}) = k \langle x_1,\ldots,x_n \rangle/(m^a),
\]
where $k\langle x_1,\ldots,x_n \rangle$ denotes the polynomial ring in $n$ non-commuting variables and $m=(x_1,\ldots,x_n)$ is the usual maximal ideal. If $a=2$ this coincides with the truncated polynomial ring in $n$ commuting variables.

The essential ingredient in the proof of Theorem \ref{t:main} is an understanding of the $S^1$-equivariant homotopy type of $B^{cy}(\Pi_{a,n})$, with $\Pi_{a,n}$ as defined in the previous section.

\begin{lemma}
We have an $S^1$-equivariant splitting
\[
 B^{cy}(\Pi_{a,n}) \cong \bigvee_w B^{cy}(\Pi_{a,n}w)[w],
\]
where the wedge sum runs over cyclic words $w$ in $x_1,\ldots,x_n$ and $B^{cy}(\Pi_{a,n})[w]$ is the subspace of $B^{cy}(\Pi_{a,n})$ defined as the geometric realization of the subsimplicial set of $N^{cy}(\Pi_{a,n})$ which in degree $q$ consists of $0$ and $(w_0,\ldots,w_q)$ with $w_0 \cdots w_q=w$ as a cyclic word.
\end{lemma}

\begin{proof}
This follows formally from the fact that the cyclic structure maps preserve the property of multiplying to $w$ or $0$.
\end{proof}

The next step is to compute the $S^1$-equivariant homotopy type of $B^{cy}(\Pi_{a,n})[w]$ for each $w$. We find the following.

\begin{lemma}
Let $w$ be a cyclic word of length $s$ in $x_1,\ldots,x_n$, and suppose $w=(w')^e$ with $w'$ an irreducible cyclic word of length $t$. Then $s=et$ and $N^{cy}(\Pi_{a,n})[w]$ is generated as a cyclic set by a single $(s-1)$-simplex with $C_e \subset C_s$ acting trivially. Moreover,
\[
 B^{cy}(\Pi_{a,n})[w] \cong S^1_+ \sma_{C_e} (\Delta^{s-1}/C_s \cdot \Delta^{s-n}).
\]
\end{lemma}

\begin{proof}
This follows formally in the same way as in \cite[Lemma 2.2.6]{HeMa97b}: Let $w=x_{i_0} \cdots x_{i_{s-1}}$. Then $N^{cy}(\Pi_{a,n};w)$ is generated as a cyclic set by the $(s-1)$-simplex $(x_{i_0},\ldots,x_{i_{s-1}})$ plus a disjoint basepoint. The cyclic set freely generated by an $(s-1)$-simplex plus a disjoint basepoint geometrically realizes to $S^1_+ \sma \Delta^{s-1}_+$, and the two types of relations given by $C_e$ acting trivially and words of length $a$ being zero give the two types of modifications.
\end{proof}

The hard part is identifying the $S^1$-equivariant homotopy type of this. Recall that $d= \lfloor \frac{s-1}{a} \rfloor$.

\begin{thm} \label{t:htpytype}
As above let $w$ be a cyclic word of length $s$ with $w=(w')^e$ with $w'$ irreducible.
\begin{enumerate}
\item If $a \nmid s$ the $S^1$-equivariant homotopy type of $B^{cy}(\Pi_{a,n})[w]$ is given by
\[
 B^{cy}(\Pi_{a,n})[w] \simeq S^1_+ \sma_{C_e} S^{V_d}
\]
\item If $a \mid s$ then $B^{cy}(\Pi_{a,n})[w]$ is $S^1$-equivariantly homotopy equivalent to the cofiber of the natural map
\[
 \bigvee_{u \in (v_a^1)^{-1}(w)} S^1_+ \sma_{C_{e'}} S^{V_d} \to S^1_+ \sma_{C_e} S^{V_d}.
\]
Here $e'$ is the natural number such that $u=(u')^{e'}$ in $S_n(a,s)$ with $u'$ irreducible.
\end{enumerate}
\end{thm}

It is worth making the second part more explicit. The wedge sum is over words $u$ modulo the equivalence relation $\sim_a$ which map to the cyclic word $w$, and there are $g = \gcd(\ell(w'), a)$ such words. Also, note that the natural number $e'$ is the same for all $u \in (v_a^1)^{-1}(w)$ and is given by $\frac{eg}{a}$.

\begin{proof}[Proof of Theorem \ref{t:htpytype}]
This now follows from Theorem \ref{t:thetaequiv}. The first case is clear. For the second case, we can write $C_a * S^{V_d} \simeq \Delta^{s-1}/C_s \cdot \Delta^{s-a}$ as a cofiber
\[
 (C_a)_+ \sma S^{V_d} \to S^{V_d} \to C_a * S^{V_d} 
\]
and after applying $S^1_+ \sma_{C_e} -$ we get a cofiber sequence
\[
 S^1_+ \sma_{C_e} ((C_a)_+ \sma S^{V_d}) \to S^1_+ \sma_{C_e} S^{V_d} \to S^1_+ \sma_{C_e} (C_a * S^{V_d}).
\]
The space on the left hand side is $S^1$-equivariantly homeomorphic to $\bigvee_g S^1_+ \sma_{C_{e'}} S^{V_d}$, and the result follows.
\end{proof}

\section{$THH$, $\TF$ and $\TC$ of the pointed monoid algebra} \label{s:THHTFTC}
Now we consider topological Hoschild homology of $A=k\langle x_1,\ldots,x_n\rangle/m^a=k(\Pi_{a,n})$. We have an $S^1$-equivariant wedge decomposition
\[
 THH(A,m) \simeq \bigvee_w THH(k) \sma B^{cy}(\Pi_{a,n})[w]
\]
where the wedge is over cyclic words. Because we are considering relative $THH$ we exclude the empty word, it corresponds to the copy of $THH(k)$ which splits off from $THH(A)$ using the splitting $k \to A \to k$.

Recall that we can compute $\TC(A,m)$ from $THH(A,m)$ in two steps. First, compute $\TF(A,m) = \lim_F THH(A,m)^{C_r}$. And second, compute $\TC(A,m)$ as the homotopy equalizer of $R, 1 : \TF(A,m) \to \TF(A,m)$. We start with $\TF(A,m)$.

\begin{prop} \label{p:TF}
We have
\[
 \TF(A,m) \simeq \bigvee_w \TF(A)[w],
\]
where the wedge is over all non-empty cyclic words. Moreover, $\TF(A)[w]$ is given as follows. Let $s$ be the length of $w$, and define $e$ by $w=(w')^e$ with $w'$ irreducible.
\begin{enumerate}
\item If $a$ does not divide $s$ then up to profinite completion we have
\[
 \TF(A)[w] \simeq \Sigma \big(THH(k) \sma S^{V_d}\big)^{C_e}.
\]
\item If $a$ divides $s$ then up to profinite completion $\TF(A)[w]$ is the homotopy cofiber of the map
\[
 \bigvee_{u \in (v_a^1)^{-1}(w)} \Sigma \big(THH(k) \sma S^{V_d}\big)^{C_{e'}} \to \Sigma \big( THH(k) \sma S^{V_d} \big)^{C_e}
\]
given on each wedge summand by the transfer associated to the inclusion of $C_e$-fixed points into $C_{e'}$-fixed points.
\end{enumerate}
\end{prop}

\begin{proof}
This follows in the same way as in \cite[Section 8]{HeMa97a}.
\end{proof}

We write $d(w)$ for $\lfloor \frac{\ell(w)-1}{a} \rfloor$, and we write $e(w)$ for $|w|_{S_n(1,N)}$. We also write $e'(u)$ for $|w|_{S_n(a,N)}$. As in \cite[Proposition 8.2]{HeMa97a} or \cite[Proposition 4.2.3]{HeMa97b} we then find the following:

\begin{prop} \label{p:TC}
Up to profinite completion $\TC(A,m)$ is the homotopy cofiber of the map
\[
 \underset{u \in S_n(a,\infty)}{\holim} \Sigma \big( THH(k) \sma S^{V_{d(u)}} \big)^{C_{e'(u)}} \to \underset{w \in S_n(1,\infty)}{\holim} \Sigma \big( THH(k) \sma S^{V_{d(w)}} \big)^{C_{e(w)}}
\]
given on each wedge summand by the transfer as in Proposition \ref{p:TF}. Here each homotopy limit is over the restriction maps.
\end{prop}

Note that we have not yet used that $k$ is a perfect field of positive characteristic; this characterization of $\TC(A,m)$ is valid for any commutative ring $k$.

Because the kernel of $A \to k$ is nilpotent, we can use the standard comparison results from \cite{Mc97, Du97, DuGoMc13} to conclude that $K(A,m)$ is given by the same cofiber. Now we can prove Theorem \ref{t:main}.

\begin{proof}[Proof of Theorem \ref{t:main}]
Suppose $k$ is a perfect field of positive characteristic. It suffices to identify the map in Proposition \ref{p:TC} with the middle map in Theorem \ref{t:main}.

Consider the group
\[
 \pi_{2q-1} \Sigma \big( THH(k) \sma S^{V_{d(w)}} \big)^{C_{e(w)}} \cong \pi_{2q-2} \big( THH(k) \sma S^{V_{d(w)}} \big)^{C_{e(w)}}.
\]
If $d(w) \leq q-1$ then this is given by $\bW_{\langle e \rangle}(k)$, otherwise it is the limit (using the restriction map $R$) over $w'$ with $w' \mid w$ of $\pi_{2q-2} \big( THH(k) \sma S^{V_{d(w')}} \big)^{C_{e(w')}}$. Putting all of that together, $\pi_{2q-1}$ of the target of the map in Proposition \ref{p:TC} is given by
\[
 \pi_{2q-1} \underset{w \in S_n(1,\infty)}{\holim} \Sigma \big( THH(k) \sma S^{V_{d(w)}} \big)^{C_{e(w)}} \cong \bW_{S_n(1,aq)}(k).
\]
The same reasoning applies to the source of the map, and we find that
\[
 \pi_{2q-1} \underset{u \in S_n(a,\infty)}{\holim} \Sigma \big( THH(k) \sma S^{V_{d(u)}} \big)^{C_{e'(u)}} \cong \bW_{S_n(a,aq)}(k).
\]
Finally, the transfer map turns into the Verschiebung in the same way as in \cite[Proposition 9.1]{HeMa97a}.
\end{proof}

\section{Calculations for $\bZ$} \label{s:integers}
For this section let $A=\bZ \langle x_1,\ldots,x_n \rangle/(m^a)$. To compute the rationalized $K$-groups of $A$ we follow the strategy used to prove \cite[Theorem 1.4]{AGHL14}. We start with one more definition.

\begin{defn}
Let $S_n(a,[M,N]) \subset S_n(a,N)$ denote the set of words of length between $M$ and $N$ (inclusive), with length divisible by $a$, modulo the relation given by cyclically permuting blocks of $a$ letters. Let $k\{S_n(a,[M,N])\}$ denote the free $k$-module on the set $S_n(a,[M,N])$.
\end{defn}

\begin{thm}
There is an exact sequence
\begin{multline*}
 0 \to K_{2q}(A,m)_\bQ \to \bQ\{S_n(a,[a(q-1)+1,aq])\} \\
 \xto{V_a^1} \bQ\{S_n(1,[a(q-1)+1,aq])\} \to K_{2q-1}(A,m)_\bQ \to 0.
\end{multline*}
Moreover, the map $V_a^1$ surjects onto $\bQ\{S_n(1,[aq,aq])\}$, so
\[
 K_{2q-1}(A,m)_\bQ \cong \bQ\{S_n(1,[a(q-1)+1,aq-1])\}.
\]
\end{thm}

Here of course $S_n(a,[a(q-1)+1,aq])=S_n(a,[aq,aq])$. It is also possible to describe $K_{2q}(A,m)_\bQ$ explicitly; it is the direct sum of copies of $\bQ$ with $g-1$ summands for each cyclic word $w$ of length $aq$, where $g=|(v_a^1)^{-1}(w)|$.

\begin{proof}
As in \cite{AGHL14}, we use that $\big( \pi_* (THH(\bZ) \sma S^{\lambda_d})^{C_e} \big)_\bQ$ has a $\bQ$ in dimension $2 \dim_\bC (\lambda_d^{C_{e'}})$ for each $e' \mid e$. Then a minor modification of the proof of \cite[Proposition 6.1]{AGHL14} shows that the homotopy cofiber in Proposition \ref{p:TC} becomes the exact sequence in the theorem.
\end{proof}

\bibliographystyle{plain}
\bibliography{b.bib}

\end{document}